\documentclass[11pt]{amsart}
\usepackage[utf8]{inputenc}
\usepackage[english]{babel}
\usepackage{amsmath, amsfonts, amssymb, amsthm, stmaryrd}

\usepackage{multirow}
\usepackage{array}
\usepackage{multicol}
\usepackage[usenames,dvipsnames]{color}
\usepackage{graphicx}
\usepackage{multirow}
\usepackage{mathrsfs}
\usepackage{array}
\usepackage{rotating}
\usepackage{multirow}
\usepackage{faktor}
\usepackage{realhats}
\usepackage{hyperref}
\usepackage{mathtools}
\usepackage{float}
\usepackage{comment}
\usepackage[backend=biber,maxbibnames=99]{biblatex}
\usepackage[ruled,linesnumbered]{algorithm2e}
\usepackage{url}

\usepackage[a4paper,top=3cm,bottom=2cm,left=3.5cm,right=3.5cm,marginparwidth=1.75cm]{geometry}

\usepackage{array, multirow}
\usepackage{pdflscape}
\usepackage{xcolor}
\usepackage{array}
\usepackage{colortbl}
\usepackage{booktabs}
\usepackage{multirow}

\newcommand{\minitab}[2][l]{\begin{tabular}{#1}#2\end{tabular}}

\newtheorem{theorem}{Theorem}[section]

\newtheorem{conjecture}[theorem]{Conjecture}
\newtheorem{corollary}[theorem]{Corollary}

\newtheorem{lemma}[theorem]{Lemma}

\newtheorem{proposition}[theorem]{Proposition}

\theoremstyle{definition}
\newtheorem{definition}[theorem]{Definition}
\newtheorem{example}[theorem]{Example}

\theoremstyle{remark}
\newtheorem{remark}[theorem]{Remark}

\newtheorem{innercustomgeneric}{\customgenericname}
\providecommand{\customgenericname}{}
\newcommand{\newcustomtheorem}[2]{%
  \newenvironment{#1}[1]
  {%
   \renewcommand\customgenericname{#2}%
   \renewcommand\theinnercustomgeneric{##1}%
   \innercustomgeneric
  }
  {\endinnercustomgeneric}
}

\newcustomtheorem{customthm}{Theorem}

\DeclareMathOperator{\qfelia}{\boldsymbol{F}}
\DeclareMathOperator{\dist}{\mathfrak{D}}

\title{Integer Factorization via Continued Fractions and Quadratic Forms}

\author[N.~Murru]{Nadir Murru} 
\address{Dipartimento di Matematica, Università di Trento, Via Sommarive 14, 38123, Povo (TN), Italy}
\email{nadir.murru@unitn.it}

\author[G.~Salvatori]{Giulia Salvatori} 
\address{Dipartimento di Scienze Matematiche L. Lagrange, Politecnico di Torino, Corso Duca degli Abruzzi 24, 10129, Torino (TO), Italy}
\email{giulia.salvatori@polito.it}

\date{}
\subjclass[2010]{11A51, 11Y05} 

\begin{document}

\begin{abstract}
We propose a novel factorization algorithm that leverages the theory underlying the SQUFOF method, including reduced quadratic forms, infrastructural distance, and Gauss composition. We also present an analysis of our method, which has a computational complexity of $O \left( \exp \left( \frac{3}{\sqrt{8}} \sqrt{\ln N \ln \ln N} \right) \right)$, making it more efficient than the classical SQUFOF and CFRAC algorithms. Additionally, our algorithm is polynomial-time, provided knowledge of a (not too large) multiple of the regulator of $\mathbb{Q}(\sqrt{N})$.
\end{abstract}

\maketitle

\section{Introduction}

The integer factorization problem is a fascinating challenge in number theory, with many important theoretical aspects and practical applications (e.g., in cryptography, where the most important public key cryptosystems are based on the hardness of solving this problem for large composite numbers). Indeed, currently, there does not exist a polynomial algorithm for factorizing integers and thus the research in this field is fundamental and active. 

To date, the most efficient algorithm known for factoring integers larger than $10^{100}$ is the General Number Field Sieve designed by Pomerance \cite{Pomerance}, with a heuristic running time of $\exp \left ( \left (\sqrt[3]{\frac{64}{9}} + o(1) \right )(\ln N)^{1/3}(\ln \ln N)^{2/3} \right ) $. 
However, for smaller numbers other algorithms perform better, such as the Quadratic Sieve and SQUare FOrm Factorization (SQUFOF).

SQUFOF algorithm (the best method for numbers between $10^{10}$ and $10^{18}$) was proposed by Shanks in \cite{Shanks} and it is based on the properties of square forms and continued fractions. 
The algorithm is discussed, for example, in \cite{buell} and \cite{cohenacourse}. A rigorous and complete description of the method and its complexity is provided in the well-regarded paper by Gower and Wagstaff \cite{GW07}, where the details are meticulously presented and the algorithm is examined in depth.
Recently, the SQUFOF algorithm has been revisited by Elia \cite{elia2019} who proposed an improvement whose complexity is based on the computation of the regulator of a quadratic field. Another improvement, which exploits a sieve inspired by the Quadratic Sieve, can be found in \cite{BW}.

The other main algorithm, which exploits the theory of continued fractions, is CFRAC \cite{LP31} which was implemented and used for factorizing large numbers (such as the 7\textsuperscript{th} Fermat number) by Morrison and Brillhart \cite{MB75}.

In this paper, we focus on the underlying theory of SQUFOF and, starting from the work by Elia \cite{elia2019}, we improve it, obtaining a novel factorization algorithm whose complexity for factoring the integer $N$ is $O \left ( \exp \left (\frac{3}{\sqrt{8}}\sqrt{\ln N \ln \ln N} \right ) \right )$, making it more efficient than the classical CFRAC and SQUFOF algorithms.
The time complexity is similar to that obtained in \cite{BW}, although their method uses a different approach.

The paper is structured as follows. Sections \ref{sec:pre}, \ref{sec:pre-even} and \ref{sec:pre-qf} introduce the notation and develop the foundational results applied in the design of the factorization algorithm. Specifically, in Section \ref{sec:pre}, we deal with the theory of continued fractions, focusing on the expansion of square roots and the properties of particular sequences arising from these expansions. 
Section \ref{sec:pre-even} analyzes the conditions under which the period of the continued fraction expansion of $\sqrt{N}$ is even and a nontrivial factor of $N$ can be found. Finally, in Section \ref{sec:pre-qf}, we introduce the tools regarding quadratic forms, including the notion of distance, the reduction operator, and the Gauss composition, focusing on the properties of particular sequences of quadratic forms used in the algorithm. In Section \ref{sec:alg}, we present and discuss all the details of our new algorithm and analyze the time complexity, highlighting also the fundamental role played by the computation of the regulator of $\mathbb Q(\sqrt{N})$. Lastly, Section \ref{sec:concl} briefly reports some conclusions.

\section{Sequences from continued fractions of quadratic irrationals} \label{sec:pre}
It is well-known that the continued fraction expansion of quadratic irrationals is periodic and in this case the Lagrange algorithm can be used for obtaining such expansion. Let us consider, without loss of generality, quadratic irrationals
\[ \alpha_0 = \cfrac{P_0 + \sqrt{N}}{Q_0}, \]
with  $P_0, Q_0, N \in \mathbb Z$, $N > 0$ not square, $Q_0 \not=0$, and $Q_0 \mid N - P_0^2$. The continued fractions expansion $[a_0, a_1, \ldots]$ of $\alpha_0$ can be obtained computing
\begin{equation}  \label{quadrirr}
\begin{cases}a_m = \lfloor \alpha_m \rfloor\\
P_{m+1} = a_m Q_m - P_m \\
Q_{m+1} = (N-P_{m+1}^2)/Q_m\end{cases}, \quad \text{where} \quad \alpha_m = \frac{P_m + \sqrt{N}}{Q_m}, \quad m \ge 0.
\end{equation}
We recall that the continued fraction expansion of $\sqrt{N}$ is periodic and has the following particular form
\begin{equation} \label{cf-sqrt}
\sqrt{N} = [a_0, \overline{a_1, a_2, a_3, \ldots, a_{\tau -1}, 2a_0}],
\end{equation}
where the sequence $(a_1, \ldots, a_{\tau-1})$ is a palindrome. 

\begin{remark}\label{remarkperiod}
\emph{
Kraitchik \cite{Kraitchik} showed that the period $\tau$ of the continued fraction expansion of $\sqrt{N}$ is upper bounded by
\(0.72 \sqrt{N} \ln N\), for \(N > 7.\)
However, the period length has irregular behavior as a function of $N$: it can assume any value from $1$, when $N = M^2 + 1$, to values greater than $\sqrt{N} \ln \ln N$ (see \cite{sierp} and \cite{longperiod}, respectively).}
\end{remark}
From now on, we always consider the continued fraction expansion of $\sqrt{N}$ as in \eqref{cf-sqrt} (i.e., quadratic irrationals with $Q_0 = 1$ and $P_0 = 0$). Let $\{p_n\}_{n \ge -1}$ and $\{q_n\}_{n \ge -1}$ be the sequences of numerators and denominators of convergents of $\sqrt{N}$, defined by
\[ p_{-1} = 1, \quad p_0 = a_0, \quad q_{-1} = 0, \quad q_0 = 1 \]
and
\[ p_m = a_m p_{m-1} + p_{m-2}, \quad q_m = a_m q_{m-1} + q_{m-2}, \quad \forall \, m \geq 1.\]
We also recall the following two properties:
\begin{equation}\label{cfmatrici} \begin{pmatrix} a_0 & 1 \cr 1 & 0 \end{pmatrix} \cdots \begin{pmatrix} a_m & 1 \cr 1 & 0 \end{pmatrix} = \begin{pmatrix} p_m & p_{m-1} \cr q_m & q_{m-1} \end{pmatrix} \quad \forall \, m \ge 0, \end{equation}
and
\begin{equation}\label{laprima}
\begin{cases} p_{\tau -2} = -a_0 p_{\tau -1} + Nq_{\tau -1}\\
q_{\tau -2} = p_{\tau -1} -a_0 q_{\tau -1}
\end{cases}.
\end{equation}
Equation \eqref{cfmatrici} follows from straightforward verification, and \eqref{laprima} is proved in \cite[329--332]{sierp}.

We examine the sequence $\left \{ \mathfrak{c}_n  \right \}_{n \ge -1}$, defined by
\[ \mathfrak{c}_m := p_m + q_m \sqrt{N}.\]
The result in the next proposition is also found in \cite{elia2019}. Here, we provide a slightly different proof for completeness.

\begin{proposition}
The sequence $\left \{ \mathfrak{c}_n  \right \}_{n \ge -1}$ satisfies the relation
\begin{equation}\label{ctau}
 \mathfrak{c}_{m + k \tau} = \mathfrak{c}_{m} \mathfrak{c}_{\tau -1}^k \quad  \text{for all } k \in \mathbb{N} \text{ and }m\ge-1. \end{equation}
 
\end{proposition}

\begin{proof}
The claimed equality is trivial for $k = 0$.
First, we prove by induction on $m$ the equality for $k=1$, and then we generalize for $k>1$. 

The case $m=-1$ is trivial, since $\mathfrak{c}_{-1} = 1$.
Now, we proceed by induction. Using the inductive 
hypothesis, we consider the following chain of equalities
\[
\begin{aligned}
\mathfrak{c}_{\tau +m+1} &= a_{m+1}\mathfrak{c}_{\tau +m} + \mathfrak{c}_{\tau +m-1} = a_{m+1} \mathfrak{c}_{m} \mathfrak{c}_{\tau -1} + \mathfrak{c}_{m-1} \mathfrak{c}_{\tau -1} \\
&= (a_{m+1} \mathfrak{c}_{m} + \mathfrak{c}_{m-1} )\mathfrak{c}_{\tau -1} = \mathfrak{c}_{m+1}\mathfrak{c}_{\tau -1} \end{aligned} \]
which concludes the proof in the case $k=1$.

For the case $k>1$ we iterate as follows:
\[\mathfrak{c}_{m + k\tau} = \mathfrak{c}_{m + (k-1)\tau}\mathfrak{c}_{\tau - 1} = \cdots = \mathfrak{c}_{m}\mathfrak{c}_{\tau - 1}^k. \]
\end{proof}

We recall that the minimal positive solution of the Pell Equation $X^2 - N Y^2 = 1$ is $(p_{\tau-1}, q_{\tau -1})$ if $\tau$ is even, and $(p_{2\tau-1}, q_{2\tau -1})$ if $\tau$ is odd. 
We denote by $R^{+}(N)$ the logarithm of $a+b \sqrt{N}$, where $(a,b)$ is the minimal positive solution of the Pell equation, 
by $R(N)$ the regulator of $\mathbb{Q}(\sqrt{N})$, and by $\mathcal{N}$ the field norm of $\mathbb{Q}(\sqrt{N})$. Moreover, given $x + y \sqrt{N} \in \mathbb{Q}(\sqrt{N})$, we denote by $\overline{x + y \sqrt{N}}$ its conjugate.

\begin{proposition}\label{proptildeqtildep}
We have
\begin{equation}\label{sviluppop}
(-1)^m P_{m+1} = p_mp_{m-1} - Nq_mq_{m-1} \quad \forall \, m \ge 0
\end{equation}
and
\begin{equation}\label{sviluppoq}
(-1)^{m+1} Q_{m+1} =p_m^2 -Nq_m^2= \mathcal{N}(\mathfrak{c}_m) \quad \forall \, m \ge -1. 
\end{equation}
\end{proposition}
\begin{proof}
The proof is straightforward by induction.
\end{proof}

Since we will exploit the sequences $\{P_n\}_{n \ge 0}$ and $\{Q_n\}_{n \ge 0}$ to factor the integer $N$, it is computationally important to bound their elements. 

\begin{proposition}[\cite{elia2019}]\label{boundsPQ}
We have
\[ 0< Q_{m} < \frac{2}{a_{m}} \sqrt{N} , \quad 0 \le P_{m} < \sqrt{N} \quad \forall \, m \ge 0.\]
\end{proposition}

The following lemma proves two equalities that are useful in Theorem \ref{teodelta}.
\begin{lemma}\label{lemmaeq8}
It holds that
\begin{equation}\label{eq8}\begin{cases}p_{\tau -m-2} = (-1)^{m-1}p_{\tau -1}p_m + (-1)^mNq_{\tau -1}q_m\\
    q_{\tau -m-2} = (-1)^{m}p_{\tau -1}q_m + (-1)^{m-1}q_{\tau -1}p_m
\end{cases}
\forall \, -1 \le m \le \tau -1.\end{equation}
\end{lemma}
\begin{proof}
Using Equation \eqref{cfmatrici} and the fact that $(a_1, \ldots, a_{\tau -1})$ is palindrome, we obtain, for all $0 \le m \le \tau - 2$,
\[\begin{aligned}
\begin{pmatrix} p_{\tau -1} & p_{\tau -2} \cr q_{\tau -1} & q_{\tau -2} \end{pmatrix} &= \begin{pmatrix} p_{\tau-m-2} & p_{\tau-m-3} \cr q_{\tau-m-2} & q_{\tau-m-3} \end{pmatrix} \begin{pmatrix} a_{\tau -m-1} & 1 \cr 1 & 0 \end{pmatrix} \cdots \begin{pmatrix} a_{\tau -1} & 1 \cr 1 & 0 \end{pmatrix} \\
&= \begin{pmatrix} p_{\tau-m-2} & p_{\tau-m-3} \cr q_{\tau-m-2} & q_{\tau-m-3} \end{pmatrix}\begin{pmatrix} a_{m+1} & 1 \cr 1 & 0 \end{pmatrix} \cdots \begin{pmatrix} a_{1} & 1 \cr 1 & 0 \end{pmatrix} \\
&= \begin{pmatrix} p_{\tau-m-2} & p_{\tau-m-3} \cr q_{\tau-m-2} & q_{\tau-m-3} \end{pmatrix}\left [ \begin{pmatrix} a_{0} & 1 \cr 1 & 0 \end{pmatrix}^{-1} \begin{pmatrix} p_{m+1}  & p_{m} \cr q_{m+1} &  q_{m} \end{pmatrix} \right ]^{T} \\
&= \begin{pmatrix} p_{\tau-m-2} & p_{\tau-m-3} \cr q_{\tau-m-2} & q_{\tau-m-3} \end{pmatrix} \begin{pmatrix} p_{m+1}  & q_{m+1} \cr p_{m} &  q_{m} \end{pmatrix} \begin{pmatrix} 0  & 1 \cr 1 &  -a_0 \end{pmatrix}. \end{aligned}
\]
Multiplying by the inverse of the matrices $\begin{pmatrix} p_{m+1}  & q_{m+1} \cr p_{m} &  q_{m} \end{pmatrix} $ and $ \begin{pmatrix} 0  & 1 \cr 1 &  -a_0 \end{pmatrix}$, and using Equation \eqref{laprima}, we obtain
\begin{equation}\label{conti}\begin{aligned}\begin{pmatrix} p_{\tau-m-2} & p_{\tau-m-3} \cr q_{\tau-m-2} & q_{\tau-m-3} \end{pmatrix} &= (-1)^m\begin{pmatrix} p_{\tau -1} & p_{\tau -2} \cr q_{\tau -1} & q_{\tau -2} \end{pmatrix}\begin{pmatrix} a_0 & 1 \cr 1 & 0 \end{pmatrix}  \begin{pmatrix} q_m  & -q_{m+1} \cr -p_m &  p_{m+1} \end{pmatrix}\\
&=(-1)^m \begin{pmatrix} Nq_{\tau-1}q_m - p_mp_{\tau-1}  & -Nq_{\tau-1}q_{m+1} \cr p_{\tau-1}q_{m} - p_{m}q_{\tau-1} &  -p_{\tau-1}q_{m+1}+q_{\tau-1}p_{m+1} \end{pmatrix}.\end{aligned}\end{equation}
The cases $m=-1$ and $m=\tau -1$ are straightforward to verify.
\end{proof}
The transformation defined by \eqref{eq8} is identified by the matrix

\[M_{\tau -1} =
\begin{bmatrix} 
-p_{\tau -1} & Nq_{\tau -1}\\
-q_{\tau -1} & p_{\tau -1} \end{bmatrix}
.\]

The results that follow in this section are those found by Elia in \cite{elia2019}, but they are further extended, approaching also the case of odd periods, and we include more detailed proofs.
The sequences $\{Q_n\}_{n \ge 0}$ and $\{P_n\}_{n \ge 1}$ are periodic of period $\tau$, where $\tau$ is the period of the sequence of partial quotients $\{ a_n\}_{n \ge 0}$ of the continued fraction expansion of $\sqrt{N}$. Further, within a period, there exist interesting symmetries.

\begin{theorem}\label{teodelta}
The sequence $\{Q_n\}_{n \ge 0}$ is periodic with period $\tau$. The elements of the first block $ \{ Q_n \}_{n=0}^{\tau}$ satisfy the
symmetry relation \[ Q_m = Q_{\tau -m}, \quad \forall \, \ 0 \le m \le \tau.\]
\end{theorem}

\begin{proof}
Using Equation \eqref{ctau}, Equation \eqref{sviluppoq}, and the fact that $\mathcal{N}(p_{\tau -1} + q_{\tau -1}\sqrt{N}) = (-1)^{\tau}$, the following chain of equalities holds for all $m \ge 0$
\[ \begin{aligned}
Q_{m+\tau} &= \left |\mathcal{N}(\mathfrak{c}_{m-1+ \tau}) \right | = \left | \mathcal{N}(\mathfrak{c}_{m-1} \mathfrak{c}_{\tau -1}) \right | = \left | \mathcal{N}(\mathfrak{c}_{m-1}) \mathcal{N}( \mathfrak{c}_{\tau -1}) \right | = Q_m, \end{aligned}\]
from which we deduce that the period of $\{Q_n\}_{n \ge 0}$ is $\tau$.

The symmetry of the sequence $\{Q_n\}_{n \ge 0}$ within the $\tau$ elements of the first period follows from Equation \eqref{eq8}.
We have
\[ \begin{aligned}
    p_{\tau -m-2}^2 -Nq_{\tau -m-2}^2 &= (p_{\tau -1}p_m -Nq_{\tau -1}q_m)^2 -N(-p_{\tau -1}q_m + q_{\tau -1}p_m)^2 \\
    &=  (p_m^2 - Nq_m^2)(p_{\tau -1}^2 - Nq_{\tau -1}^2) \\
    &=  (-1)^{\tau}(p_m^2 - Nq_m^2),
\end{aligned}\]
implying that $(-1)^{\tau -m-1}Q_{\tau -m-1} = (-1)^{\tau}(-1)^{m+1} Q_{m+1}$ for all $-1 \le m \le \tau -1$. 
\end{proof}

\begin{theorem}\label{teoomega}
The sequence $\{P_n\}_{n \ge 1}$ is periodic with period $\tau$. The elements of the first block $\{ P_m \}_{m=1}^{\tau}$ satisfy the symmetry relation
\begin{equation}\label{simmetriap}
P_{\tau -m+1} = P_{m}, \quad \forall \, \ 1 \le m \le \tau.\end{equation}
\end{theorem}

\begin{proof}
The periodicity of the sequence $\{P_n\}_{n \ge 1}$ follows from the properties expressed by Equation \eqref{ctau} and Equation \eqref{sviluppop}, noting that
\[ \begin{aligned}(-1)^m P_{m+1} &= \frac{1}{2}(\mathfrak{c}_{m}\overline{\mathfrak{c}_{m-1}} + \overline{\mathfrak{c}_{m}}\mathfrak{c}_{m-1}) \\
&=\frac{1}{2} \left (\frac{\mathfrak{c}_{m+\tau}}{\mathfrak{c}_{\tau - 1}} \frac{\overline{\mathfrak{c}_{m-1+\tau}}}{\overline{\mathfrak{c}_{\tau - 1}}} + \frac{\overline{\mathfrak{c}_{m + \tau}}}{\overline{\mathfrak{c}_{\tau - 1}}} \frac{\mathfrak{c}_{m-1+\tau}}{\mathfrak{c}_{\tau - 1}} \right ) \\
&= (-1)^{\tau} (-1)^{m+\tau} P_{m+1+\tau}.\end{aligned}\]
The next chain of equalities proves the symmetry property
\[\begin{aligned}(-1)^{\tau -m-1}P_{\tau -m} &= p_{\tau -m-1}p_{(\tau -1)-m-1} - Nq_{\tau-1-m}q_{(\tau -1) -m-1} \\
  &= -(p_{\tau -1}p_m - Nq_{\tau -1}q_m)(p_{\tau -1}p_{m-1} - Nq_{\tau -1}q_{m-1}) \\
  &+ N(p_{\tau -1}q_m - q_{\tau -1}p_m)(p_{\tau -1}q_{m-1} - p_{m -1}q_{\tau-1}) \\
    &= -(p_{\tau -1}^2 -Nq_{\tau -1}^2)(p_mp_{m-1} -Nq_mq_{m-1}) \\
    &= -(-1)^{\tau}(p_mp_{m-1} -Nq_mq_{m-1}) \\
    &= (-1)^{\tau +1} (-1)^{m} P_{m+1},
\end{aligned}\]
where, in the second-to-last equality, we used Equation \eqref{sviluppop}.
\end{proof}

Note that $M_{\tau -1}^2 = (-1)^{\tau}I_2$, with $I_2$ the identity matrix, and, if $\tau$ is even, the eigenvalues of $M_{\tau -1}$ are $\lambda_0 =1$ and $\lambda_1 =-1$, with eigenvectors
\begin{equation} \label{eingen}
U^{(h)} = \left [ \frac{p_{\tau -1} + \lambda_h}{d} , \frac{q_{\tau -1}}{d}\right ]^{T},
\end{equation}
where $d = \gcd(p_{\tau -1} + \lambda_h, q_{\tau -1}) $ for $h \in \left \{ 0,1 \right \}$.

\begin{theorem}\label{teoremaimportante}
If the period $\tau$ of the continued fraction expansion of $\sqrt{N}$ is even, a factor of $2N$ is located at positions $\frac{\tau}{2} + j\tau$ with $j = 0, 1, \ldots$, in the sequence $\{Q_n\}_{n \ge 0}$.
\end{theorem}

\begin{proof}
It is sufficient to consider $j = 0$, due to the periodicity of $\{Q_n\}_{n \ge 0}$. Since $\tau$ is even, $M_{\tau -1}$ is involutory and has eigenvalues $\lambda_0 =1$ and $\lambda_1 =-1$ with corresponding eigenvectors shown in \eqref{eingen}.
Considering Equation \eqref{eq8} written as
\[ \begin{bmatrix}
    p_{\tau - j-2} \\
    q_{\tau - j-2}
\end{bmatrix} = (-1)^{j-1}M_{\tau -1} \begin{bmatrix}
    p_{j} \\
    q_{j}
\end{bmatrix},\]
we see that $V^{(j)} = \left [ p_j, q_j \right ]^T$ is an eigenvector of $M_{\tau -1}$, of eigenvalue $(-1)^{j-1}$, if and only if $j$ satisfies the condition $\tau -j-2 = j$, that is $j=\frac{\tau -2}{2}= \tau_0$. From the comparison of $V^{(j)}$ and $U^{(h)}$, we have
\[ p_{\tau_0} = \frac{p_{\tau -1} + (-1)^{\tau_0}}{d} \quad  q_{\tau_0} = \frac{q_{\tau -1}}{d},\]
where the equalities are fully motivated because $\text{gcd}	\left ( p_{\tau_0}, q_{\tau_0}\right )=1$, recalling that $d = \gcd(p_{\tau -1} + (-1)^{\tau_0},q_{\tau -1})$. Direct computation yield
\[ (-1)^{\tau_0+1} Q_{\tau_0+1} = \frac{(p_{\tau -1} + (-1)^{\tau_0 -1})^2 - Nq_{\tau -1}^2}{d^2} = 2\frac{(-1)^{\tau_0}p_{\tau -1} +1}{d^2}, \]
which can be written as $p_{\tau_0}^2 - Nq_{\tau_0}^2 = 2(-1)^{\tau_0} \frac{p_{\tau_0}}{d}$. Dividing this equality by $2\frac{p_{\tau_0}}{d}$ we have
\[ \frac{d p_{\tau_0}}{2}- N\frac{1}{\frac{2p_{\tau_0}}{d}}q_{\tau_0}^2 = (-1)^{\tau_0}.\]
Noting that $\text{gcd}\left ( p_{\tau_0}, q_{\tau_0} \right )=1$, it follows that $ \frac{2p_{\tau_0}}{d}$ is a divisor of $2N$, i.e. $Q_{\tau_0+1} = Q_{\tau/2} \mid 2N$.
\end{proof}

In the case where $\tau$ is odd, we can state the following two results.

\begin{theorem}\label{oddperiodsumsquares}
Let $N$ be a positive integer such that the continued fraction expansion of $\sqrt{N}$ has an odd period $\tau$.
The representation of $N$ as a sum of two squares is given by $N = a^2 + b^2$, where $a = Q_{(\tau +1)/2}$ and $b = P_{(\tau +1)/2}$.
\end{theorem}
\begin{proof}
Since $\tau$ is odd, by the anti-symmetry in the sequence $  
 \{ Q_n\}_{n=0}^{\tau -1}$, we have $Q_{(\tau +1)/2} = Q_{(\tau -1)/2}$, so that the quadratic form $Q_{(\tau -1)/2} X^2 + 2P_{(\tau +1)/2}XY - Q_{(\tau +1)/2} Y^2$ has discriminant $4P_{(\tau +1)/2}^2 - 4Q_{(\tau +1)/2} Q_{(\tau -1)/2} = 4N$, which shows the assertion.
\end{proof}
From this, we can deduce a result similar to that in Theorem \ref{teoremaimportante} for the case of an odd period.

\begin{corollary}
Let $N>0$ be a composite nonsquare integer such that the continued fraction expansion of $\sqrt{N}$ has odd period $\tau$. If $-1$ is a quadratic nonresidue modulo $N$, then $Q_{(\tau +1)/2}$ contains a nontrivial factor of $N$.
\end{corollary}
\begin{proof}
Using the previous theorem, $N = Q_{(\tau +1)/2}^2 + P_{(\tau +1)/2}^2$, and so $P_{(\tau +1)/2}^2 \equiv - Q_{(\tau +1)/2}^2 \pmod{N}$. If $\gcd (N, Q_{(\tau +1)/2}) = 1$, then $Q_{(\tau +1)/2}^{-1} \pmod{N}$ exists. Therefore, $\left (Q^{-1}_{(\tau +1)/2} P_{(\tau +1)/2} \right )^2 \equiv -1 \pmod{N}$, and so $-1$ is a quadratic residue modulo $N$.
\end{proof}

\begin{lemma}\label{lemmauguaglianza}
The following identity holds
\[ \frac{\sqrt{N} + P_{m+1}}{Q_{m+1}} = - \frac{p_{m-1} - q_{m-1} \sqrt{N}}{p_{m} - q_{m} \sqrt{N}} \quad \forall \, m \ge 0.\]
\end{lemma}
\begin{proof}
The proof is straightforward.
\end{proof}

The following result will be used in the proof of Theorem \ref{distanzaciclopari}.

\begin{lemma}\label{lemmagamma1.44}
If $\tau$ is even, defining $\gamma$ as
\[ \gamma = \prod_{m=0}^{\tau -1} (\sqrt{N} + P_{m+1}),\]
we have $\frac{\gamma}{\overline{\gamma}} =(p_{\tau -1} + q_{\tau -1} \sqrt{N})^2 = \mathfrak{c}_{\tau -1}^2$.
If $\tau$ is odd, defining $\omega$ as
\[ \omega = \prod_{m=0}^{2\tau -1} (\sqrt{N} + P_{m+1}),\]
we have $\frac{\omega}{\overline{\omega}} =(p_{2\tau -1} + q_{2\tau -1} \sqrt{N})^2 = \mathfrak{c}_{2\tau -1}^2$.
\end{lemma}

\begin{proof}
We provide a proof for the case $\tau$ even; the odd case follows the same procedure. We have
\[\begin{aligned}
\frac{\gamma}{\overline{\gamma}} &= \prod_{m=0}^{\tau -1} \frac{\sqrt{N} + P_{m+1}}{-\sqrt{N} + P_{m+1}} \\
&= \prod_{m=0}^{\tau -1} \frac{(\sqrt{N} + P_{m+1})^2}{ P_{m+1}^2 -N} \\
&= \prod_{m=0}^{\tau -1} \frac{(\sqrt{N} + P_{m+1})^2}{-Q_{m+1} Q_{m}}.
\end{aligned}\]
Noting that $\prod_{m=0}^{\tau -1} - Q_m Q_{m+1} = \prod_{m=0}^{\tau -1} - Q_{m+1}^2 = \prod_{m=0}^{\tau -1} Q_{m+1}^2$ due to the periodicity of the sequence $\{Q_m\}_{m \ge 0}$ and the parity of $\tau$, we deduce that $\frac{\gamma}{ \overline{\gamma}}$ is a perfect square. From Lemma \ref{lemmauguaglianza}, it follows that the base of the square giving $\frac{\gamma}{\overline{\gamma}}$ is
\[ \begin{aligned}
\prod_{m=0}^{\tau -1} \frac{\sqrt{N} + P_{m+1}}{Q_{m+1}} &= \prod_{m=0}^{\tau -1} - \frac{p_{m-1} - q_{m-1} \sqrt{N}}{p_{m} - q_{m} \sqrt{N}} \\
&= \frac{p_{-1} - q_{-1}\sqrt{N}}{p_{\tau -1} - q_{\tau -1}\sqrt{N}} \\
&= p_{\tau -1} + q_{\tau -1}\sqrt{N}.
\end{aligned}\]
Therefore, 
\begin{equation} \label{quarta}
\prod_{m=0}^{\tau -1} \frac{\sqrt{N} + P_{m+1}}{Q_{m+1}} = p_{\tau -1} + q_{\tau -1}\sqrt{N} = \mathfrak{c}_{\tau -1},\end{equation}
and in conclusion $\frac{\gamma}{\overline{\gamma}} = \mathfrak{c}_{\tau -1}^2$. 

Similarly, if $\tau$ is odd, we have
\begin{equation} \label{quinta}
\prod_{m=0}^{2\tau -1} \frac{\sqrt{N} + P_{m+1}}{Q_{m+1}} = p_{2\tau -1} + q_{2\tau -1}\sqrt{N} = \mathfrak{c}_{2\tau -1}.\end{equation}
\end{proof}

\section{Even period and nontrivial factor} \label{sec:pre-even}
In this section, we establish conditions on the integer $N$ and its factors to ensure that the period $\tau$ is even and that $Q_{\tau/2} \ne 2$. First, we address the problem of guaranteeing an even period, and then, under this assumption, we derive conditions for the existence of a nontrivial factor of $N$ (i.e., $Q_{\tau/2} \ne 2$). Subsequently, we turn our attention to the case where $N$ is an RSA modulus.

According to a classical result on the Pell equation, the period $\tau$ of the continued fraction expansion of $\sqrt{N}$ is even if and only if the negative Pell equation
\begin{equation}\label{negpelleq}
X^2 - NY^2 = -1
\end{equation}
has no solution. Based on this, we derive the following sufficient condition on the factors of $N$ for $\tau$ to be even.

\begin{proposition}\label{parityperiodprop}
Let $N>0$ be a nonsquare integer. If $N$ is divided by a prime $p \equiv 3 \pmod{4}$, then the period $\tau$ of the continued fraction expansion of $\sqrt{N}$ is even.
\end{proposition}

\begin{proof}
Suppose that \eqref{negpelleq} has an integral solution $(u,v)$. Then, $u^2 \equiv -1 \pmod{N}$, and so $u^2 \equiv -1 \pmod{p}$. This means that $-1$ is a quadratic residue modulo $p$, but this cannot be possible since $p \equiv 3 \pmod{4}$.
\end{proof}

As we can see in the example below, this is not a necessary condition.

\begin{example}
Let $N = 5^2 \cdot 17 \cdot 37 = 15725$, which is not divisible by any prime $p \equiv 3 \pmod{4}$. The period of the continued fraction expansion of $\sqrt{15725}$ is $10$.
\end{example}

Determining the parity of the period when no primes congruent to $3 \pmod{4}$ divide $N$ is a challenging open problem.
Recently, Koymans and Pagano \cite{koymans2022stevenhagen} proved the following theorem, originally conjectured by Stevenhagen in \cite{stevenhagen1993number}. For further results in this area see \cite{Chan_Koymans_Milovic_Pagano}, \cite{fouvry2010negative} and \cite{Fouvry_Kluners_period}.

\begin{theorem}[\cite{koymans2022stevenhagen}]\label{stevenhagenproved}
Let $\mathcal{D} = \{ N \in \mathbb{N} \mid N \text{ squarefree and not divisible by primes }p\equiv 3 \pmod{4}\}$, $\mathcal{D}^{-} = \{ N \in \mathcal{D} \mid \text{\eqref{negpelleq} has an integral solution} \}$, $(\mathcal{D})_{\le X}= \{ N \in \mathcal{D} \mid N \le X \}$ and $(\mathcal{D}^{-})_{\le X}= \{ N \in \mathcal{D}^{-} \mid N \le X \}$.
We have
\[ \lim_{X\rightarrow \infty} \frac{\# (\mathcal{D}^{-})_{\le X}}{\# (\mathcal{D})_{\le X}} = 1- \alpha, \]
where 
\[\alpha = \prod_{j \text{ odd}} (1-2^{-j}) = 0.41942244117951...\]
This result does hold when restricted to odd/even numbers.
\end{theorem}

We now examine the case $N = pq$ and provide sufficient conditions on $p$ and $q$ for determining the parity of $\tau$.

\begin{proposition}[\cite{rippon2004even}]\label{propevenperiod}
If $N=rs$, then $\tau$ is even if and only if one of the following two conditions holds:
\begin{enumerate}
    \item $rX^2 - sY^2 = \pm 2$, with $X$ and $Y$ odd;
    \item $r,s \ne 1$ and $rX^2 - sY^2 = \pm 1$, with $X$ and $Y$ integers.
\end{enumerate}
\end{proposition}

Using the above proposition, we can provide sufficient conditions on $p$ and $q$ for odd period.

\begin{proposition}\label{proppqodd}
Let $N = pq$, where $p$ and $q$ are primes congruent to $1 \pmod{4}$. If $\left( \frac{p}{q} \right)=-1$, then the period $\tau$ of the continued fraction expansion of $\sqrt{N}$ is odd.
\end{proposition}
\begin{proof}
The first equation of Proposition \ref{propevenperiod} cannot have integral solutions, since in this case $r \equiv s \equiv 1 \pmod{4}$, and so $rX^2 - sY^2 \equiv 0 \pmod{4}$. 
By Hasse--Minkowski theorem, $\left( \frac{p}{q} \right)=1$ is a necessary condition for the solubility of $pX^2 - qY^2 = \pm 1$ in $\mathbb{Q}$. Therefore, if $\left( \frac{p}{q} \right)=-1$ then, by Proposition \ref{propevenperiod}, the period $\tau$ is odd.
\end{proof}

The following proposition, proved by Dirichlet in \cite{dir1834}, gives us sufficient conditions on $p$ and $q$ for even period.

\begin{proposition}[\cite{dir1834}]\label{proppqeven}
Let $N = pq$, where $p$ and $q$ are primes congruent to $1 \pmod{4}$. If $\left( \frac{p}{q} \right) = 1$ and $\left( \frac{p}{q} \right)_4 \left( \frac{q}{p} \right)_4 = -1$, then the period of the continued fraction expansion of $\sqrt{N}$ is even.
\end{proposition}
The conditions in Proposition \ref{proppqodd} and Proposition \ref{proppqeven} are sufficient but not necessary, as showed in the following example.

\begin{example}
Consider $N = 5 \cdot 89 = 445$, then $\left( \frac{5}{89} \right)	= 1$ and the period of the continued fraction of $\sqrt{445}$ is $5$.\\
Consider $N = 13 \cdot 53 = 689$, then $\left( \frac{13}{53} \right)	= 1$, $\left( \frac{13}{53} \right)_4 \left( \frac{53}{13} \right)_4 = 1$ and the period of the continued fraction of $\sqrt{689}$ is $2$.
\end{example}

Determining the parity of the period when $N = pq$ and $p \equiv q \equiv 1 \pmod{4}$ remains a challenging problem. The following conjecture is the version of Theorem \ref{stevenhagenproved} restricted to integers with exactly two prime factors.

\begin{conjecture}[\cite{stevenhagen1993number}]\label{conjecure1}
Let $\mathcal{D}_2 = \{ N=pq \mid p \equiv q \equiv 1 \pmod{4} \}$, $\mathcal{D}_2^{-} = \{ N \in \mathcal{D}_2 \mid \tau \equiv 1 \pmod{2} \}$, $(\mathcal{D}_2)_{\le X}= \{ N \in \mathcal{D}_2 \mid N \le X \}$ and $(\mathcal{D}_2^{-})_{\le X}= \{ N \in \mathcal{D}_2^{-} \mid N \le X \}$. Then, the following limit 
\[ \lim_{X\rightarrow \infty} \frac{\# (\mathcal{D}_2^{-})_{\le X}}{\# (\mathcal{D}_2)_{\le X}} \]
exists and it is equal to $\frac{2}{3}$.
\end{conjecture}

Cremona--Odoni \cite{cremonaodoni} and Stevenhagen \cite{stevenhagen1993number} studied the problem when the number of prime divisors equals a fixed integer $t \ge 1$.

We derive conditions on $N$ for a proper factorization, specifically conditions ensuring that $Q_{\tau/2} \neq 2$. The following theorem, proved by Mollin in \cite{mollin}, provides necessary and sufficient conditions, expressed in terms of Diophantine equations, for $\tau$ to be even and $Q_{\tau/2} = 2$.

\begin{theorem}[\cite{mollin}]\label{teogenelia}
The following statements are equivalent for $N > 2$.
\begin{enumerate}
    \item $X^2 - NY^2 = \pm 2$ is solvable, indicating that at least one of the equations $X^2 - NY^2 = 2$ and $X^2 - NY^2 =-2$ has a solution.
    \item $\tau$ is even and $Q_{\tau/2}=2$.
\end{enumerate}
\end{theorem}
This result implies that, if $\tau \equiv 0 \pmod{2}$ and the two Diophantine equations 
\begin{equation}\label{pellduepm}
X^2 - N Y^2 = 2 \quad \text{and} \quad X^2 - N Y^2 = -2 \end{equation}
have no solutions, then the central term $Q_{\tau/2} \ne 2$, and so it contains a proper factor of $N$. The following result, due to Yokoi, and presented in \cite{yokoi_1994}, gives us sufficient conditions for the insolubility of the two equations in \eqref{pellduepm}. 

\begin{proposition}[\cite{yokoi_1994}]
For any positive nonsquare integer $N$, if the Diophantine equation $X^2 - N Y^2 = \pm 2$ has an integral solution, then 
\[N \equiv 2 \pmod{4} \quad \text{or} \quad N \equiv 3 \pmod{4}.\]
\end{proposition}

Hence, the following set of integers guarantees parity of the period and $Q_{\tau/2} \ne 2$
\[\mathcal{F} = \left\{ N \in \mathbb{N} \mid N \equiv 1 \pmod{4} \text{ and } \exists \, p \text{ prime}, \, p \mid N \text{ such that } p \equiv 3 \pmod{4} \right\}.\]

\begin{proposition}
Let $N \in \mathbb{N}$ such that exist two primes $p \equiv 5 \pmod{8}$ and $q \equiv 4 \pmod{4}$ such that $pq \mid N$. Then, the period $\tau$ of the continued fraction expansion of $\sqrt{N}$ is even and $Q_{\tau/2} \ne 2$.
\end{proposition}
\begin{proof}
By Proposition \ref{parityperiodprop} we deduce the parity of the period. If one of the two Diophantine equations $X^2 - NY^2 = 2$ and $X^2 - NY^2 =-2$ admits an integral solution, then $2$ or $-2$ is a quadratic residue modulo $p$, which is absurd. We conclude using Theorem \ref{teogenelia}.
\end{proof}

We now focus on RSA moduli $N = pq$ and summarize the results in Table \ref{tabellapq}.

\begin{corollary}
Let $N = pq$, where $p$ and $q$ are primes, and let $\tau$ be the period of the continued fraction expansion of $\sqrt{N}$.
\begin{enumerate}
\item If $p \equiv q \equiv 3 \pmod{4}$, then $\tau$ is even and $Q_{\tau/2}$ contains a nontrivial factor of $N$.
\item If $p \equiv 5 \pmod{8}$ and $q \equiv 3 \pmod{4}$, then $\tau$ is even and $Q_{\tau/2}$ contains a nontrivial factor of $N$.
\item If $p \equiv q \equiv 1 \pmod{4}$ and $\tau$ is even, then $Q_{\tau/2}$ contains a nontrivial factor of $N$.
\end{enumerate}
\end{corollary}
Finally, in the case $N=pq$ with $p \equiv 1 \pmod{8}$ and $q \equiv 3 \pmod{4}$, we have a sufficient condition for a trivial factorization, proved by Ji in \cite{ji2005diophantine}.

\begin{proposition}[\cite{ji2005diophantine}]\label{propgrafi}
Let $N = pq$, where $p \equiv 1 \pmod{8}$ and $q \equiv 3 \pmod{4}$ are primes. If $\left( \frac{p}{q} \right) = -1$, then one of the following two Diophantine equations
\[
X^2 - NY^2 = 2 \quad \text{or} \quad X^2 - NY^2 = -2
\]
has an integral solution.
\end{proposition}

The following example demonstrates that the conditions in Proposition \ref{propgrafi} are sufficient but not necessary.

\begin{example}
Let $N = 17 \cdot 43 = 731$. Then, $\left( \frac{17}{43} \right) = 1$, the period $\tau$ of the continued fraction expansion of $\sqrt{731}$ is $2$, and $Q_{\tau/2} = 2$.
\end{example}

Table \ref{tabellapq} summarizes the results described above for the case of RSA moduli.

\begin{table}[H]\label{tabellapq}
\centering
\setlength{\aboverulesep}{-0.45pt} \setlength{\belowrulesep}{-0.45pt}
\centering
\begin{tabular}{*{4}{|c}|}
\hline
\multirow{1}{*}{$p \pmod{8}$}& {$q \pmod{8}$}& {$\tau \pmod{2}$} & {$Q_{\tau/2}$}\\[1pt]
\hline
3 & 3 & \multirow{3}{*}{$0$} & \multirow{3}*{\minitab[c]{ $\ne 2$ }} \\
3 & 7 & &  \\
7 & 7 & &  \\[1pt]
\hline 
5 & 7 & \multirow{2}{*}{$0$} & \multirow{2}*{\minitab[c]{ $\ne 2$}} \\
5 & 3 &  &  \\[1pt]
\hline 
1 & 7 & \multirow{2}{*}{$0$} & \multirow{2}*{\minitab[c]{If $\left( \frac{p}{q} \right) = -1$, then $=2$}} \\
1 & 3 &  &  \\[1pt]
\hline 
1 & 1 & \multirow{3}*{\minitab[c]{\vspace{1pt}If $\left( \frac{p}{q} \right)_4 \left( \frac{q}{p} \right)_4 = -1$, then $0$\\ If $\left( \frac{p}{q} \right) = -1$, then $1$}} &  \multirow{3}*{\minitab[c]{If $\tau$ even, then $\ne 2$}} \\
1 & 5 &  &  \\
5 & 5 & &  \\[1.5pt]
\hline
\end{tabular}
\caption{Conditions for the parity of the period and nontrivial factorization for $N = pq$.}
\label{tab:label_tabella}
\end{table}

\section{Quadratic forms}\label{sec:pre-qf}

An overview of binary quadratic forms can be found in \cite{buell}. 
\begin{definition}
A \textit{binary quadratic form} is a polynomial $F(X, Y) = aX^2+bXY+cY^2$, with $a, b, c \in \mathbb{Z}$.
The \textit{matrix associated} with $F$ is
\[ M_F =\begin{bmatrix}  a & b/2 \\ b/2 & c \end{bmatrix}.\]
\end{definition}
We abbreviate a binary quadratic form with coefficients $a,b,c$ as $(a,b,c)$.

\begin{definition}\label{defequivalent}
Two quadratic forms $F$ and $F'$ are \textit{equivalent} if there exists a matrix $C \in \mathbb{Z}^{2 \times 2}$ such that 
\[ M_{F'} = C^{T} M_F C\]
and $\det(C) = \pm 1$. If $\det(C) = 1$, the forms are \textit{properly equivalent} and we write $F \sim F'$.
\end{definition}

\begin{definition}
The \textit{discriminant} $\Delta$ of a quadratic form $(a,b,c)$ is $\Delta = b^2 -4ac$. We define $\mathbb{F}_{\Delta}$ as the set of all quadratic forms of discriminant $\Delta$.
\end{definition}
The discriminant is an invariant for the equivalence relation of quadratic forms $\sim$: if $F \in \mathbb{F}_{\Delta}$, and $F \sim F'$, then $F' \in \mathbb{F}_{\Delta}$.

\begin{definition}\label{defreduced}
A quadratic form $(a,b,c)$, with positive discriminant $\Delta= b^2 -4ac$ is \textit{reduced} if
\begin{equation} \label{reduced}
\quad  \left | \sqrt{\Delta} -2  \left | a \right | \right | < b < \sqrt{\Delta}.\end{equation}
\end{definition}

Given a quadratic form $F$, it is always possible to find a reduced quadratic form equivalent to $F$. In the following we are going to prove it giving a reduction algorithm on quadratic forms of positive nonsquare discriminant. To do so, we first need the following definition.

\begin{definition}\label{defrho}
For any form $F = (a, b, c)$ with $ac \ne 0$ of discriminant $\Delta$, a nonsquare positive integer, we define the \textit{standard reduction operator} $\rho$ by
\[\rho(a, b, c) = \left (c, r(-b, c), \frac{r(-b, c)^2 - \Delta}{4c} \right ), \]
where $r(-b, c)$ is defined to be the unique integer $r$ such that $r +b \equiv 0 \pmod{2c}$ and
\[ -|c| < r \le |c| \quad \text{if} \quad \sqrt{\Delta}< |c|,\]
\[\sqrt{\Delta} - 2|c| < r < \sqrt{\Delta} \quad \text{if} \quad |c| < \sqrt{\Delta}.\]
$\rho(F)$ is called the \textit{reduction} of $F$. The \textit{inverse reduction operator} is defined by
\[\rho^{-1}(a,b,c) = \left ( \frac{r(-b, c)^2 - \Delta}{4c}, r(-b, a), a \right ).\]
\end{definition}
We denote $\rho^n(F)$ the result of $n$ applications of $\rho$ on $F$. The identities $\rho(\rho^{-1}(F)) = \rho^{-1}(\rho(F)) = F$ hold when $F$ is reduced. We point out the fact that $(a,b,c)\sim \rho(a,b,c)$ through the transformation given by the matrix 
\[\begin{bmatrix} 0 & -1 \\ 1 & t \end{bmatrix},\]
where $r(-b,c)=-b+2ct$.
The proof of the following fundamental proposition can be found in \cite[p.~264]{cohenacourse}.
\begin{proposition}[\cite{cohenacourse}]\label{propreductiontre}
\ 
\begin{enumerate}
    \item The number of iterations of $\rho$ which are necessary to reduce a form $(a, b, c)$ is at most $2 + \left  \lceil \log_2 ( |c|/ \sqrt{\Delta} ) \right \rceil$.
    \item If $F = (a, b, c)$ is a reduced form, then $\rho(a, b, c)$ is again a reduced form.
\end{enumerate}
\end{proposition}

\begin{remark}\label{remarkfinitenumber}
\emph{
If $(a, b, c)$, of discriminant $\Delta >0$, is reduced, then $\left | a \right |$, $b$ and $\left | c \right |$ are less than $\sqrt{\Delta}$, and $a$ and $c$ are of opposite signs (\cite[p.~262]{cohenacourse}). This implies that the number of reduced quadratic forms of discriminant $\Delta$ is finite.}
\end{remark}

\begin{definition}\label{defadj}
Two forms $F(X, Y) = aX^2+bXY+cY^2$ and $F'(X, Y) = a'X^2+b'XY+c'Y^2$ are \textit{adjacent} if $c = a'$ and $b + b' \equiv 0 \pmod{2c}$.
\end{definition} 

Given a reduced quadratic form $F$, there exists a unique reduced quadratic form equivalent to $F$ and adjacent to $F$. This form is $\rho(F)$. As we have seen in Remark \ref{remarkfinitenumber}, there exists a finite number of reduced quadratic forms of positive discriminant $\Delta$, and so this process eventually repeats, forming a \textit{cycle}. The significant aspect of this is that the cycle consists of all the reduced forms equivalent to the first form, as proved in \cite[109--113]{jacobson2008solving}.

\begin{definition}\label{defprincipalcycle}
We call the \textit{principal form} the unique reduced form of discriminant $\Delta$ having as first coefficient $1$. It is denoted by $\underline{1}$ and the cycle in which it lies is called the \textit{principal cycle}. 
\end{definition}

\begin{definition}
Let $\mathbf{\Upsilon}= \{ \qfelia_n \}_{n \ge 0}$ be the sequence of binary quadratic forms defined as
\[ \qfelia_m(X,Y) = (-1)^{m}Q_{m} X^2 + 2 P_{m+1} XY + (-1)^{m+1}Q_{m+1}Y^2, \quad \text{for }m \ge0.\]
\end{definition}

\begin{remark}
\emph{
If $\tau$ is even, then the sequence $\mathbf{\Upsilon}$ is periodic of period $\tau$, and if $\tau$ is odd $\mathbf{\Upsilon}$ is periodic of period $2\tau$. This is due to Theorem \ref{teodelta} and Theorem \ref{teoomega}.}
\end{remark}

\begin{definition}\label{definizionecompositioneformecapitolodue} 
Let $F=(a_1, b_1, c_1)$ and $G=(a_2, b_2, c_2)$ two quadratic forms having same discriminant $\Delta$. The \textit{Gauss composition} of $F$ and $G$ is
\begin{equation}\label{formulacomposizione}
F \circ G = (a_3, b_3, c_3) = \left ( d_0 \frac{a_1 a_2}{n^2}, b_1 + \frac{2 a_1}{n} \left ( \frac{s(b_2 - b_1)}{2} - c_1 v\right ), \frac{b_3^2 - \Delta}{4 a_3} \right ),
\end{equation}
where $\beta = (b_1 + b_2)/2$, $n = \gcd(a_1, a_2, \beta)$, $s, u, v$ such that $a_1 s + a_2 u + \beta v = n$, and $d_0 = \gcd(a_1,a_2,\beta, c_1,c_2,(b_1-b_2)/2)$. Although the composition is not unique, all compositions of given forms $F$ and $G$ are equivalent.
\end{definition}
We remark that all quadratic forms in $\mathbf{\Upsilon}$ have the same discriminant $\Delta =4N$, where $N>0$ is the nonsquare integer we want to factorize. This implies that for all $(a,b,c) \in \mathbf{\Upsilon}$, we have
\[\Delta \equiv b^2 \equiv b \pmod{2},\]
and so $b \equiv 0 \pmod{2}$. Therefore, the value $\beta$ in the Definition \ref{definizionecompositioneformecapitolodue} is an integer.
A quadratic form $F=(a,b,c)$ is \textit{primitive} if $\gcd( a, b, c ) = 1$. As proved in the next proposition, the forms in $\mathbf{\Upsilon}$ are primitive.

\begin{proposition}
The forms $\qfelia_n$ are primitive for all $n\ge0$.
\end{proposition}
\begin{proof}
We prove the statement by induction on $n$. \\
Base step ($n=0$): The base step is proved recalling that $Q_0=1$. \\
Inductive step ($n \Rightarrow n+1$): The result follows from these equalities:
\[\begin{aligned}\gcd(Q_{n}, 2 P_{n+1}, Q_{n+1}) &= \gcd(Q_n, 2(a_nQ_n -P_n), (N-P_{n+1}^2)/Q_n) \\
&= \gcd(Q_n,-2P_n, Q_{n-1} -a_{n}^2Q_n +2a_nP_n) \\
&= \gcd(Q_{n},  -2P_n, Q_{n-1}) =1. \end{aligned} \]
\end{proof}
This implies that in the Gauss composition of two elements of $\mathbf{\Upsilon}$, the coefficient $d_0$ is always equal to $1$.

Using the definition of $\rho$ and Equation \eqref{quadrirr}, it is straightforward to prove that
\[\rho^n(\qfelia_0)=\qfelia_n.\]
The form $\qfelia_0$ is reduced and equal to $\underline{1}$, so the quadratic forms in $\mathbf{\Upsilon}$ are reduced, and $\mathbf{\Upsilon}$ is the principal cycle. Moreover, for any pair of forms $\qfelia_n, \qfelia_m \in \mathbf{\Upsilon}$, their Gauss composition $\qfelia_n \circ \qfelia_m$ is equivalent to $\qfelia_0$. This follows from the property proved by Gauss in \cite[Article 237-239]{gauss1966}: if $F \sim G$, then $H \circ F \sim H \circ G$, for all quadratic forms $F, G, H$ having same discriminant. In particular, we obtain
\[\qfelia_n \circ \qfelia_m \sim \qfelia_n \circ \qfelia_0 \sim \qfelia_n \sim \qfelia_0,\] 
using also the fact that $\qfelia_0 \circ \qfelia_n \sim \qfelia_n$ for all $n \ge 0$. Consequently, applying the Gauss composition to any couple of elements of $\mathbf{\Upsilon}$, followed by a sufficient number of applications of $\rho$ to obtain a reduced form, results in an element of $\mathbf{\Upsilon}$.

As mentioned previously, we are interested in quickly finding the coefficient $Q_{\tau/2}$ when $\tau$ is even, or $Q_{(\tau + 1)/2}$ when $\tau$ is odd. Consequently, we aim to determine the quadratic form $\qfelia_{\tau/2}$, or $\qfelia_{(\tau - 1)/2}$, in an efficient manner (i.e. with time complexity $O(\ln(N)^{\alpha})$ with $\alpha$ constant). The value of $\tau$ could be too large (see Remark \ref{remarkperiod}), so we need a way to make longer jumps within the principal cycle. 
As we will see, Gauss composition, followed by the reduction, will allow us to make long jumps in $\mathbf{\Upsilon}$. To estimate the length of these jumps we use the (well-known) \textit{infrastructural distance} $\delta$. A comprehensive definition and detailed description of distance is provided in \cite[279--283]{cohenacourse}.  

\begin{definition}
Given a quadratic form $F=(a,b,c)$ of discriminant $\Delta>0$, the \textit{infrastructural distance} $\delta$ of $F$ and $\rho(F)$ is
\[ \delta(F, \rho(F)) = \frac{1}{2} \ln \left | \frac{b + \sqrt{\Delta}}{b - \sqrt{\Delta}} \right |. \]
Given $n>0$, the  distance $\delta$ of $F$ and $\rho^n(F)$ is
\[ \delta(F,\rho^n(F)) = \sum_{i=1}^n \delta(\rho^{i-1}(F),\rho^{i}(F)). \]
\end{definition}

We now restrict ourselves to forms in the principal cycle. We then have the following proposition.

\begin{proposition}[\cite{cohenacourse}]\label{propquadrato}
Let $\qfelia_n$ and $\qfelia_m$ be two reduced forms in the principal cycle, and let $\qfelia_0$ be the principal form. Then, if we define $G = \qfelia_n \circ \qfelia_m$, $G$ may not be reduced, but let $\qfelia_r$ be a (non unique) form obtained from $G$ by the reduction algorithm, i.e. by successive
applications of $\rho$. Then we have
\[\delta(\qfelia_0,\qfelia_r) = \delta(\qfelia_0,\qfelia_n) + \delta(\qfelia_0,\qfelia_m) + \delta(G,\qfelia_r),\]
and furthermore,
\begin{equation}\label{distanzariduzione}
\left | \delta(G,\qfelia_r) \right | < 2 \ln \Delta,\end{equation}
where $\Delta$ is the discriminant of these forms.
\end{proposition}
The above proposition follows from the property that $\delta$ is exactly additive under composition before any reductions are made (see \cite[p.~281]{cohenacourse}) and the estimation of the bound for $\left | \delta(G,\qfelia_r)\right |$ discussed in Section 12 of \cite{Lenstra}.

The next proposition is fundamental for a computational point of view. Indeed, given $\qfelia_i$, $\qfelia_j$, with $i < j$, and their distance $\delta(\qfelia_i, \qfelia_j)$, it gives an estimation of $j-i$. In particular, if $\delta(\qfelia_i, \qfelia_j) = D$, then $\frac{2D}{\ln (4N)}<j-i< \frac{2D}{\ln 2} + 1$.

\begin{proposition}[\cite{Lenstra}]\label{remdistrho}
Let $F \in \mathbb{F}_{\Delta}$ reduced. The following two bounds hold:
\begin{enumerate}
    \item $\delta(F, \rho(F)) < \frac{1}{2} \ln \Delta$,
    \item $\delta( F , \rho^2( F )) > \ln 2$, and the same holds for $\rho^{-1}$.
\end{enumerate}
\end{proposition}

In our case, the discriminant of the forms in $\mathbf{\Upsilon}$ is $\Delta=4N$, where $N$ is an odd nonsquare integer. Theorem \ref{distanzaciclopari} (proved also in \cite{elia2019}) and Theorem \ref{distanzaciclodispari} show that the distance between quadratic forms can be considered modulo
\[R^{+}(N) = \begin{cases} \ln(p_{\tau -1} + q_{\tau -1} \sqrt{N})= \ln(\mathfrak{c}_{\tau -1}) & \text{if }\tau \equiv 0 \pmod{2}\\
\ln(p_{2\tau -1} + q_{2\tau -1} \sqrt{N})= \ln(\mathfrak{c}_{2\tau -1}) & \text{if }\tau \equiv 1 \pmod{2}\\
\end{cases} .\]

\begin{theorem}\label{distanzaciclopari}
If $\tau$ is even, the distance $\delta(\qfelia_0, \qfelia_{\tau})$ (the distance of a period) is exactly equal to $\ln(\mathfrak{c}_{\tau -1})$ and the distance $\delta(\qfelia_0,\qfelia_{\tau/2})$ is exactly equal to $\frac{1}{2}\delta(\qfelia_0, \qfelia_{\tau})$.
\end{theorem}
\begin{proof}
The distance between $\qfelia_{\tau}$ and $\qfelia_{0}$ is the summation
\[d(\qfelia_{0},\qfelia_{\tau}) = \sum_{i=0}^{\tau -1} d(\qfelia_i, \qfelia_{i+1}) = \sum_{i=1}^{\tau} \frac{1}{2} \ln \left( \frac{\sqrt{N} + P_i}{\sqrt{N} - P_i} \right) = \frac{1}{2} \ln \left ( \prod_{i=1}^{\tau} \frac{\sqrt{N} + P_i}{\sqrt{N} - P_i} \right ). \]
Recalling that $N- P_i^2 = Q_i Q_{i-1} >0$, and taking into account the periodicity of the sequence $\{Q_n\}_{n \ge 0}$, the last expression can be written with rational denominator as
\[\frac{1}{2} \ln \left ( \prod_{i=1}^{\tau} \frac{(\sqrt{N} + P_i)^2}{Q_i Q_{i-1}} \right ) = \frac{1}{2} \ln \left ( \prod_{i=1}^{\tau} \frac{(\sqrt{N} +P_i)^2}{ Q_i^2} \right ) = \ln \left ( \prod_{i=1}^{\tau} \frac{\sqrt{N} + P_i}{ Q_i} \right ).\]
The conclusion follows from Equation \eqref{quarta}. The equality $d(\qfelia_{0}, \qfelia_{\tau/2}) = \frac{1}{2}d(\qfelia_{0},\qfelia_{\tau})$ is an immediate consequence of the symmetry of the sequence $\{P_n\}_{n \ge 1}$ within a period.
\end{proof}

We now give a similar result for the case of odd period.

\begin{theorem}\label{distanzaciclodispari}
If $\tau$ is odd, the distance $\delta(\qfelia_0, \qfelia_{2\tau})$ (the distance of a period) is exactly equal to $\ln(\mathfrak{c}_{2\tau -1})$ and the distance $\delta(\qfelia_0, \qfelia_{\tau})$ is equal to $\ln(\mathfrak{c}_{2\tau -1})/2$.
\end{theorem}
\begin{proof}
The distance between $\qfelia_{2\tau}$ and $\qfelia_{0}$ is the summation
\[d(\qfelia_{0},\qfelia_{2\tau}) = \sum_{i=0}^{2\tau -1} d(\qfelia_i, \qfelia_{i+1}) = \sum_{i=1}^{2\tau} \frac{1}{2} \ln \left( \frac{\sqrt{N} + P_i}{\sqrt{N} - P_i} \right) = \frac{1}{2} \ln \left ( \prod_{i=1}^{2\tau} \frac{\sqrt{N} + P_i}{\sqrt{N} - P_i} \right ). \]
Recalling that $N- P_i^2 = Q_i Q_{i-1} >0$, and taking into account the periodicity of the sequence $\{Q_n\}_{n \ge 0}$, the last expression can be written with rational denominator as
\[\frac{1}{2} \ln \left ( \prod_{i=1}^{2\tau} \frac{(\sqrt{N} + P_i)^2}{Q_i Q_{i-1}} \right ) = \frac{1}{2} \ln \left ( \prod_{i=1}^{2\tau} \frac{(\sqrt{N} +P_i)^2}{ Q_i^2} \right ) = \ln \left ( \prod_{i=1}^{2\tau} \frac{\sqrt{N} + P_i}{ Q_i} \right ).\]
The conclusion follows from Equation \eqref{quinta} and the periodicity of $\{ P_n \}_{n \ge 1}$.
\end{proof}

\begin{corollary}\label{corollariodistanzadispari}
If $\tau$ is odd, the distance $\delta(\qfelia_0, \qfelia_{(\tau -1)/2})$ (the distance of a target form) is equal to $\ln(\mathfrak{c}_{2\tau -1})/4 + O(\ln (N))$.
\end{corollary}

\begin{proof}
From the previous theorem, we know that $\delta(\qfelia_0, \qfelia_{\tau})=\ln(\mathfrak{c}_{2\tau -1})/2$. Moreover, using the symmetry \eqref{simmetriap}, we obtain the following equality
\[\delta(\qfelia_0,\qfelia_{\tau}) = 2 \delta(\qfelia_0,\qfelia_{(\tau - 1)/2}) + \frac{1}{2} \ln \left( \frac{\sqrt{N} + P_{(\tau + 1)/2}}{\sqrt{N} - P_{(\tau + 1)/2}} \right).\]
Therefore,
\[\begin{aligned}\delta(\qfelia_0,\qfelia_{(\tau - 1)/2}) &= \frac{\delta(\qfelia_0,\qfelia_{\tau})}{2} - \frac{1}{4}\ln \left( \frac{\sqrt{N} + P_{(\tau + 1)/2}}{\sqrt{N} - P_{(\tau + 1)/2}} \right) \\
&= \frac{1}{4} \ln(\mathfrak{c}_{2\tau -1}) - \frac{1}{4}\ln \left( \frac{\sqrt{N} + P_{(\tau + 1)/2}}{\sqrt{N} - P_{(\tau + 1)/2}} \right), \end{aligned}\]
and so
\[ \left | \delta(\qfelia_0,\qfelia_{(\tau - 1)/2}) - \frac{1}{4} \ln(\mathfrak{c}_{2\tau -1}) \right | \le \frac{1}{4} \ln (4N).\]\end{proof}

The following remark is fundamental from a computational point of view, because it provides an upper bound on the distance of a cycle in $\mathbf{\Upsilon}$.

\begin{remark}\label{remarkregulator}
\emph{
We have that $R^{+}(N) = n R(N)$, with $n \le 6$. Hua, in \cite[p.~329]{hua1982introduction}, proves that \[R(N) \le \begin{cases}\sqrt{N}\left (\frac{1}{2}\ln N + 1 \right ) & \text{if }N \equiv 1 \pmod{4} \\
2\sqrt{N}\left (\frac{1}{2}\ln (4N) + 1 \right ) & \text{if }N \equiv 3 \pmod{4}\end{cases} \]
and thus $R^{+}(N) = O(\sqrt{N} \ln N)$. However, we do not know which is the largest value that $R(N)$ can attain as a function of $N$. It is conjectured that there exists an infinite set of values of $N$ such that $R(N) \gg \sqrt{N} \ln \ln N$ (see \cite{boundsregulatorjac} for large-scale numerical experiments and more details).}
\end{remark}

\section{The factorization algorithm}\label{sec:alg}
In this section, we present our factorization method.
The integer $N>0$ to be factorized is odd, nonsquare and composite.
In the first part of this section, we describe the method and provide the pseudocodes (Algorithms \ref{alg_nostro_1} and \ref{alg_nostro}). We then prove the correctness of our approach and analyze its computational cost.
This method is a modification of the one presented by Elia \cite{elia2019}. 
We assume that $R^{+}(N)$ has been preliminarily computed. In the final part of this section, we mention a method for computing an integer multiple of $R^{+}(N)$.

To simplify the notation, we introduce the following definition.

\begin{definition}
Given two forms $\qfelia_{n}, \qfelia_{m} \in \mathbf{\Upsilon}$, the \textit{giant step} of $\qfelia_{n}$ and $\qfelia_{m}$ is the composition \[\qfelia_{n}\bullet \qfelia_{m} = \rho^r(\qfelia_{n} \circ \qfelia_{m}),\] realized through the Gauss composition $\qfelia_{n} \circ \qfelia_{m}$, followed by the minimum number $r$ of reduction operations $\rho$ to obtain a reduced form. The notation $\qfelia_{n}^{t}$ represents $t$ successive applications of the giant step of $\qfelia_{n}$ with itself, i.e., $\qfelia_{n} \bullet \cdots \bullet \qfelia_{n}$ (repeated $t$ times).
\end{definition}

We define our method for both the even-period and odd-period cases. To enhance readability, we define the following quantity
\[\dist(N) = \begin{cases} R^{+}(N)/2 & \text{if } \tau \equiv 0 \pmod{2}\\
R^{+}(N)/4 & \text{if } \tau \equiv 1 \pmod{2}\\\end{cases},\]
which represents the distance of the quadratic form we want to reach (or an approximation of it). Indeed, if $\tau \equiv 0 \pmod{2}$, then $\delta(\qfelia_0, \qfelia_{\tau/2}) = R^{+}(N)/2 = \dist(N)$ and if $\tau \equiv 1 \pmod{2}$, then $\delta(\qfelia_0, \qfelia_{(\tau - 1)/2}) = R^{+}(N)/4 + O(\ln(N)) = \dist(N) + O(\ln(N))$, using Theorem \ref{distanzaciclopari} and Corollary \ref{corollariodistanzadispari}.

We distinguish two cases: $R^{+}(N) \le (\ln N)^2$ and $R^{+}(N) > (\ln N)^2$. In the first case, we compute $\qfelia_i = \rho^i(\qfelia_0)$ until a nontrivial factor of $N$ is found among their coefficients, if such a factor exists. By Proposition \ref{remdistrho}, the number of reduction steps $\rho$ is at most $\frac{2 \delta(\qfelia_0, \qfelia_{\tau/2})}{\ln 2} + 1 = \frac{R^{+}(N)}{\ln 2} + 1 = O((\ln N)^2)$ when the period is even, and $\frac{2 \delta(\qfelia_0, \qfelia_{(\tau - 1)/2})}{\ln 2} + 1 \le \frac{R^{+}(N)}{2 \ln 2} + \frac{\ln (4N)}{2 \ln 2} + 1 = O((\ln N)^2)$, when the period is odd. If the number of iterations exceeds $\frac{R^{+}(N)}{\ln 2} + \frac{\ln (4N)}{2 \ln 2} + 1$ the procedure is stopped: our algorithm cannot find a factor of $N$. The pseudocode for this method is given in Algorithm \ref{alg_nostro_1}.

If $R^{+}(N) > (\ln N)^2$ we proceed in the following way.

\begin{enumerate}
    \item \textbf{First phase}: In this phase we compute an approximation of $\qfelia_{\tau/2}$, if $\tau$ is even, or of $\qfelia_{(\tau - 1)/2}$, if $\tau$ is odd. By an approximation of $F \in \mathbf{\Upsilon}$, we mean a form $G \in \mathbf{\Upsilon}$ such that either $\delta(F, G)$ or $\delta(G, F)$ is small.

    Starting from $\qfelia_0$, we compute the forms $\qfelia_i$ in the principal cycle, for $i=0, \ldots, \ell$, until $\delta(\qfelia_0, \qfelia_{\ell}) \ge 2\ln (4N) +1$ and $\delta(\qfelia_0, \qfelia_{\ell}) \le 4\ln N$ (this is possible using Proposition \ref{remdistrho} and the definition of distance). Then, we compute the quadratic forms $\qfelia_{\ell}^{2^i}$, using giant steps, and their exact distance $d_i = \delta(\qfelia_0, \qfelia_{\ell}^{2^i})$, using Proposition \ref{propquadrato}, for $i=1, \ldots, t$, with $t$ such that $d_{t-1} \le \dist(N) < d_t$. We point out the fact that $d_{i+1}>d_i$ for all $i \ge 0$. 
    Then, using the forms $\qfelia_{\ell}, \ldots, \qfelia_{\ell}^{2^{t-1}}$, we compute $\Bar{F}$, which approximates $\qfelia_{\tau/2}$ if the period is even, or $\qfelia_{(\tau - 1)/2}$ otherwise.
    To do so, first we set $\Bar{F}= \qfelia_{\ell}^{2^{t -1}}$ and $\Bar{d} = d_{t-1}$. 
    Then, we start by computing $\Bar{d} + d_{t-2}$: if it is smaller or equal than $\dist(N)$, we update $\Bar{F}$ with $\Bar{F} \bullet \qfelia_{\ell}^{2^{t -2}}$ and $\Bar{d}$ with $\Bar{d} + d_{t-2}$. 
    We iterate this procedure for $i=t -3, \ldots, 0$ by computing $\Bar{d}+ d_{i}$, comparing it with $\dist(N)$, and, if it is smaller or equal, updating $\Bar{F}$ with $\Bar{F} \bullet \qfelia_{\ell}^{2^{i}}$ and $\Bar{d}$ with $\Bar{d} + d_{i}$. 

    \item \textbf{Second phase}: Starting from $\Bar{F}$, we iterate the operators $\rho$ and $\rho^{-1}$ until a factor of $N$ is found. An upper bound on the number of iterations of $\rho$ and $\rho^{-1}$ needed to find a factor (both in the case of even and odd period) is given by: 
    \[\Psi(R^{+}(N), N) := \frac{2}{\ln 2} \left ( 4 \ln (4N) \log_{2} \left ( \frac{R^{+}(N)}{2} \right ) + \frac{33}{4} \ln (4N) \right ) + 1.\]
    This bound derives from the results demonstrated later in this section. 
\end{enumerate}
A priori, we do not know the parity of $\tau$, so we proceed as follows (as outlined in Algorithm \ref{alg_nostro}). First, we run the procedure assuming $\tau \equiv 0 \pmod{2}$, which implies $\dist(N) = R^{+}(N)/2$. If, at the end of the second phase, after $\Psi(R^{+}(N), N)$ steps, a factor is not found, we then try again assuming $\tau \equiv 1 \pmod{2}$, which implies $\dist(N) = R^{+}(N)/4$. If, even in this case, no factor is found after $\Psi(R^{+}(N), N)$ steps during the second phase, the output is $-1$: our method cannot factor $N$.

\IncMargin{1.5em}
\begin{algorithm}[h]
	\caption{Our method, assuming $R^{+}(N) \le (\ln N)^2$ known}\label{alg_nostro_1}
	\SetKwData{Left}{left}
	\SetKwData{This}{this}
	\SetKwData{Up}{up}
	\SetKwFunction{Union}{Union}
	\SetKwFunction{FindCompress}{FindCompress}
	\SetKwInOut{Input}{Input}
	\SetKwInOut{Output}{Output}
        \SetKwComment{Comment}{$\triangleright$ }{ }
	\Input{An odd, composite nonsquare integer $N>0$; $R^{+}(N)$.}
	\Output{A factor of $N$ if the method is applicable; $-1$ otherwise.}
	\BlankLine
	$a_0 \gets \lfloor \sqrt{N} \rfloor$, $F \gets (1, 2a_0, a_0^2 -N) = (Q_0, 2P_1, -Q_1)$
    
    $i \gets 1$, $i_{max} \gets \frac{R^{+}(N)}{\ln 2} + \frac{\ln (4N)}{2 \ln 2}  + 1$

    \While{$i \le i_{max}$}{

        \If{$\gcd (Q_i, N)>1$}{
            \Return $\gcd (Q_i, N)$
        }

    $i \gets i+1$
    
    $F \gets \rho(F)= ((-1)^{i-1}Q_{i-1}, 2P_i, (-1)^{i}Q_i)$

    }
    
    \Return $-1$

\end{algorithm}
\DecMargin{1.5em}

\clearpage

\IncMargin{1.5em}
\begin{algorithm}[H]
	\caption{Our method, assuming $R^{+}(N) > (\ln N)^2$ known}\label{alg_nostro}
	\SetKwData{Left}{left}
	\SetKwData{This}{this}
	\SetKwData{Up}{up}
	\SetKwFunction{Union}{Union}
	\SetKwFunction{FindCompress}{FindCompress}
	\SetKwInOut{Input}{Input}
	\SetKwInOut{Output}{Output}
        \SetKwComment{Comment}{$\triangleright$ }{ }
	\Input{An odd, composite nonsquare integer $N>0$; $R^{+}(N)$.}
	\Output{A factor of $N$ if the method is applicable; $-1$ otherwise.}
	\BlankLine

    $i_{max} \gets \frac{2}{\ln 2} \left ( 4 \ln (4N) \log_{2} \left ( \frac{R^{+}(N)}{2} \right ) + \frac{33}{4} \ln (4N) \right ) + 1$

    $a_0 \gets \lfloor \sqrt{N} \rfloor$, $G_0 \gets (1, 2a_0, a_0^2 -N)$, $d_0 \gets \frac{1}{2} \ln \left | \frac{a_0 + \sqrt{N}}{a_0 - \sqrt{N}} \right |$, $\mathcal{F} \gets \emptyset$

    \While{$d_0 < 2 \ln (4N) +1$}{
    $G_0 \gets \rho(G_0)= (a,b,c)$

    $d_0 \gets d_0 + \frac{1}{2} \ln \left | \frac{b + \sqrt{4N}}{b - \sqrt{4N}} \right |$

    }
 
    \For{$j=1,2$}{

    $\dist(N) \gets R^{+}(N)/2^j$, $\mathcal{F} \gets \{ (G_0 , d_0) \}$, $i \gets 0$

    \While{$d_i \le \dist(N)$}{
    $G_{i+1} \gets G_{i} \circ G_{i} = (a,b,c)$

    $d_{i+1} \gets 2 d_{i}$

    \While{$G_{i+1}$ not reduced}{
        $d_{i+1} \gets d_{i+1} + \frac{1}{2} \ln \left | \frac{b + \sqrt{4N}}{b - \sqrt{4N}} \right |$

        $G_{i+1} \gets \rho(G_{i+1}) = (a,b,c)$
    }

    $\mathcal{F} \gets \mathcal{F} \cup \{ (G_{i+1} , d_{i+1}) \}$

    $i \gets i+1$
    }

    $t \gets i$, $\Bar{d} \gets d_{t -1}$, $\Bar{F} \gets G_{t -1}$
    
    \For{$i=t-2, \ldots, 0$}{
        \If{$\Bar{d} + d_i \le \dist(N)$}{
            $\Bar{F} \gets \Bar{F} \bullet G_i$ 
            
            $\Bar{d} \gets \Bar{d} + d_i$
        }
    }  

    $H \gets \rho(\Bar{F})=(a,b,c)$, $K \gets \rho^{-1}(\Bar{F})=(d,e,f)$

    \For{$i=0, \ldots, i_{max}$}{
        \uIf{$\gcd(c, N)>1$}{
            \Return $\gcd(c, N)$
        }
        \uElseIf{$\gcd(f, N)>1$}{
            \Return $\gcd(f, N)$
        }
        \Else{
            $H \gets \rho(H)= (a,b,c)$
        
            $K \gets \rho^{-1}(K)= (d,e,f)$
        }
    }
  
    }

    \Return $-1$

\end{algorithm}
\DecMargin{1.5em}

\ 

In what follows, we present two propositions that play a fundamental role in the analysis of our method for the case $R^{+}(N)> (\ln N)^2$.
The first shows that $t$ (the number of powers of $G_0$) is always ``small", the second proves that our approximation of $\qfelia_{\tau/2}$, if $\tau$ even, or $\qfelia_{(\tau - 1)/2}$ if $\tau$ is odd, is good.

\begin{proposition}\label{propt}
The value of $t$ in Algorithm \ref{alg_nostro} is at most $\left \lceil \log_{2} \dist(N) \right \rceil$.
\end{proposition}
\begin{proof}
Using Proposition \ref{propquadrato} and the above notation, we have that
\[\delta(\qfelia_0, G_i) > 2\delta(\qfelia_0, G_{i-1}) - 
2\ln (4N) \quad \forall \, i >0, \]
and so
\[\begin{aligned}
\delta(\qfelia_0, G_i) &> 2^i\delta(\qfelia_0, G_{0}) - 
2 \sum_{k=0}^{i-1} 2^k \ln(4N) \\
&= 2^i\delta(\qfelia_0, G_{0}) - 2(2^i -1)\ln(4N) \\
&\ge 2^i(2 \ln(4N) +1) - 2(2^i -1)\ln(4N) \\
&= 2^i + 2\ln(4N).
\end{aligned}\]
Therefore, for $i \ge \left \lceil \log_2 \dist(N) \right \rceil$, we have $\delta(\qfelia_0, G_i) > \dist(N)$.
\end{proof}

This proposition implies that $t = O(\ln N)$, thanks to Remark \ref{remarkregulator}.

\begin{proposition}\label{distanzafbar}
Let $\Bar{F}$ be the quadratic form obtained at the end of the second phase of the method (using the notation of Algorithm \ref{alg_nostro}). Then, the following holds
\[\left | \dist(N) - \delta(\qfelia_0, \Bar{F}) \right | = O((\ln N)^2).\]
\end{proposition}

\begin{proof}
In the for loop at line 20 of Algorithm \ref{alg_nostro}, we have at most $t-1$ giant steps. Therefore, using the previous proposition and Proposition \ref{propquadrato}, we have that
$$\left | \Bar{d} - \delta(\qfelia_0, \Bar{F}) \right | = O((\ln N)^2).$$
Now, we prove that $\left | \Bar{d}- \dist(N) \right | = O((\ln N)^2)$.
We define $I \subseteq \{0, \ldots, t-1 \}$ the set of indexes of the distances $d_0, \ldots, d_{t-1}$ that appear in the computation of $\Bar{d}$, i.e.
\[\Bar{d} = \sum_{i \in I} d_i.\]
We distinguish two cases:
\begin{itemize}
    \item Case $0 \notin I$: Then we have $\Bar{d} + d_0 > \dist(N)$, and so
    \[\Bar{d} \le \dist(N) < \Bar{d} + d_0, \]
    from which we deduce that
    \[0 \le \dist(N) - \Bar{d} < d_0 \le 4 \ln N.\]
    \item Case $0 \in I$: Let $j = \min \{ i \in 	\mathbb{N} \mid i \notin I \}$, then $1 \le j \le t-2$. We have that
    \[\begin{aligned} \Bar{d} + d_0 &= \sum_{i \in I\setminus \{ 0 \}} d_i + 2 d_0 \\
    &= \sum_{i \in I\setminus \{ 0 \}} d_i + d_1 + O(\ln N) \\
    &= \sum_{i \in I\setminus \{ 0,1 \}} d_i + 2d_1 + O(\ln N) \\
    &= \sum_{i \in I\setminus \{ 0,1 \}} d_i + d_2 + O(\ln N) \\
    &\ \ \vdots \\
    &= \sum_{i \in I\setminus \{ 0, \ldots, j-1\}} d_i + d_j + \gamma(N) \\
    \end{aligned} \]
    where $\gamma(N)=O((\ln N)^2)$.
    By construction, we have that \[\sum_{i \in I\setminus \{ 0, \ldots, j-1\}} d_i + d_j > \dist(N),\] from which
    \[\Bar{d} \le \dist(N) \le \Bar{d} + d_0 - \gamma(N),\]
    and so 
    \[0 \le \dist(N) - \Bar{d} \le d_0 - \gamma(N) = O((\ln N)^2).\]
\end{itemize}
Therefore,
\[ \left |\dist(N) - \delta(\qfelia_0, \Bar{F}) \right | \le \left |\dist(N) - \Bar{d} \right | + \left | \Bar{d} - \delta(\qfelia_0, \Bar{F}) \right | = O((\ln N)^2). \]
\end{proof}
Therefore, if $\tau$ is even, then $\left | \delta(\qfelia_0, \qfelia_{\tau /2}) - \delta(\qfelia_0, \Bar{F}) \right | = O((\ln N)^2)$. If $\tau$ is odd, we have that \[\begin{aligned}\left | \delta(\qfelia_0, \qfelia_{(\tau - 1 )/2}) - \delta(\qfelia_0, \Bar{F}) \right | &= \left | \dist(N) -\delta(\qfelia_0, \Bar{F}) \right | +  O(\ln N) = O((\ln N)^2)\end{aligned}\]
since $\delta(\qfelia_0, \qfelia_{(\tau - 1 )/2}) = \dist(N) + O(\ln N)$ by Corollary \ref{corollariodistanzadispari}.
We analyze the computational cost of this algorithm, assuming that $R^{+}(N)> (\ln N)^2$ is preliminarily computed. The cost of the computation of the form $G_0$ such that $\delta(\qfelia_0, G_0) \ge 2 \ln(4N)+1$ is at most $\frac{2(2 \ln(4N)+1)}{\ln 2} + 1$, using Proposition \ref{remdistrho}.

We then analyze the cost of the while loop at line 9. The number $t$ of the giant steps is at most $\left \lceil \log_{2} \dist(N) \right \rceil = O(\ln N)$, proved in Proposition \ref{propt}. Each giant step requires: $O(\ln N)$ elementary operations for the extended Euclidean algorithm, used to compute $s$, $u$ and $v$ in the Gauss composition, and at most $O(\ln N)$ applications of $\rho$. 
Indeed, if we apply the Gauss composition of two forms in $\mathbf{\Upsilon}$, using the extended Euclidean algorithm, we obtain $(a,b,c)$ such that $\left | c \right | = O(N^4)$. This follows from Proposition \ref{boundsPQ} and the classical bounds on the solution of Bézout's identity via the extended Euclidean algorithm. Hence, by Proposition \ref{propreductiontre}, the number of applications of $\rho$ is $O(\ln N)$. Therefore, the cost of the while loop at line 9 is $O((\ln N)^2)$. 

The computation of $\Bar{F}$ requires at most $O((\ln N)^2)$ steps: at most $O(\ln N)$ giant steps, and at most $O(\ln N)$ application of $\rho$ for each giant step.

Finally, as proved in Proposition \ref{distanzafbar}, the distance between the approximation $\Bar{F}$ and $\qfelia_{\tau/2}$, or $\qfelia_{(\tau -1)/2}$, is at most $O((\ln N)^2)$, and so, using again the fact that, for each reduced form $F$, $\delta(F, \rho^2(F))> \ln 2$, are needed only $O((\ln N)^2)$ applications of $\rho$ to reach $\qfelia_{\tau/2}$ or $\qfelia_{(\tau -1)/2}$.

In conclusion, the method has a computational complexity of $O((\ln N)^2)$.
It is remarked that the cost of elementary arithmetic operations (i.e. additions, subtractions, multiplications and divisions of big integers) and logarithm valuations are not counted.

\begin{remark}\label{remarkricerca}
\emph{
Using Proposition \ref{propt}, Proposition \ref{distanzafbar}, Corollary \ref{corollariodistanzadispari}, and Proposition \ref{remdistrho}, it is possible to derive the following upper bound on the number of iterations of $\rho$ and $\rho^{-1}$ in the second phase of the method
\[\Psi(R^{+}(N), N) = \frac{2}{\ln 2} \left ( 4 \ln (4N) \log_{2} \left ( \frac{R^{+}(N)}{2} \right ) + \frac{33}{4} \ln (4N) \right ) + 1.\]
This bound holds for both the cases when $\tau$ is even and $\tau$ is odd.}
\end{remark}

\begin{remark}\label{remarkricerca2}
\emph{
We point out that it is not necessary to have $R^{+}(N)$ precomputed; it is sufficient to have an integer multiple of it: $R'(N) = k R^{+}(N)$, with $k \in \mathbb{N}$. For simplicity, we describe how the method is modified in the case of an even period. In this case, running Algorithm \ref{alg_nostro} with $R'(N)$ instead of $R^{+}(N)$ is equivalent to considering the principal cycle with multiplicity $k$ (i.e. $k$ times the principal cycle). Our target form is located in the middle of some period of distance $R^{+}(N)/2$, so:
\begin{enumerate}
    \item if $k$ is odd, a factor of $2N$ is found (as coefficient of a form) at the position at distance $\frac{k R^{+}(N)}{2}$, from the beginning;
    \item if $k$ is even, the quadratic form $\qfelia_{\tau-1}$ is found in a position at distance $\frac{k R^{+}(N)}{2}$ (which reveals a posteriori that $k$ is even); in this case, the procedure can be repeated, targeting the form at position at distance $\frac{k R^{+}(N)}{4}$ from $\qfelia_0$. Again, either a factor of $2N$ is found, or $k$ is found to be a multiple of $4$. Clearly the process can be iterated $h$ times until $\frac{k R^{+}(N)}{2^h}$ is an odd multiple of $R^{+}(N)$, and a factor of $2N$ is found.
\end{enumerate}
If $k$, as a function of $N$, is $O(N^{\alpha})$ with $\alpha$ constant, then the computational cost of the algorithm does not change. Indeed, in this case, the value of $t$ in Algorithm \ref{alg_nostro} is at most $\lceil \log_2 ( k\dist(N) ) \rceil = O(\ln N)$ (see Proposition \ref{propt}), and the number of iterations of $\rho$ and $\rho^{-1}$ is at most $\Psi(kR^{+}(N), N) = O((\ln N)^2)$. The odd-period case follows similarly.}
\end{remark}

As outlined in Remark \ref{remarkricerca2}, we are focused on studying and researching methods to efficiently calculate or approximate $R^{+}(N)$, or one of its integer multiples that is ``not too large". In particular, we are seeking a method that efficiently computes $kR^{+}(N)$, where $k$, as a function of $N$, is $O(N^{\alpha})$, with $\alpha$ constant. Since $R^{+}(N) = n R(N)$, with $n \le 6$, our problem is equivalent to finding an efficient algorithm that computes (a multiple of) the regulator of $\mathbb{Q}(\sqrt{N})$. The method due to Vollmer, described in \cite{vollmer}, currently has the best known complexity. 
It is a Monte Carlo algorithm that computes $R(N)$ in time $O \left ( \exp \left (  \frac{3}{\sqrt{8}} \sqrt{\ln N \ln \ln N} \right ) \right )$, assuming the Generalized Riemann Hypothesis (GRH). 
Other methods for computing the regulator are detailed in \cite{jacobson2008solving}.

In conclusion, using this approach for the precomputation of $R(N)$ and the algorithm previously described, we have obtained a factorization method of (conjectured) time complexity $O \left ( \exp \left (\frac{3}{\sqrt{8}}\sqrt{\ln N \ln \ln N} \right ) \right )$, which is more efficient than CFRAC and SQUFOF.

\section{Conclusions}\label{sec:concl}
We proposed a novel factorization algorithm which is polynomial-time, provided knowledge of a (not too large) multiple of the regulator of $\mathbb{Q}(\sqrt{N})$, or an accurate approximation of it. 
The problem of computing the regulator $R(N)$ lies in $\mathcal{NP} \cap \text{co-}\mathcal{NP}$, under the assumptions of the GRH and the Extended Riemann Hypothesis (ERH), as shown in \cite[Section 13.6]{jacobson2008solving}.

A natural direction for advancing our method involves studying techniques for finding good approximations of $kR(N)$, with $k$ being a positive integer.
A common approach in this line of research involves the use of the analytic class number formula
\begin{equation}\label{analyticclassnumberformula}h(N)R(N) = \sqrt{D(N)} L(1, \chi_{D(N)})\end{equation}
where $N$ is squarefree, $D(N)$ is the discriminant of $\mathbb{Q}(\sqrt{N})$, $h(N)$ is the class number of $\mathbb{Q}(\sqrt{N})$, $\chi_{D(N)}$ is the Kronecker symbol $\left ( \frac{D(N)}{\cdot} \right )$, and $L(1, \chi_{D(N)})$ is the Dirichlet $L$-function defined by
\[ L(1, \chi_{D(N)})= \sum_{n=1}^{\infty} \frac{1}{n} \left ( \frac{D(N)}{n} \right ) .\]

Determining precise bounds on $R(N)$ is a difficult problem, closely connected to the Cohen--Lenstra heuristics \cite{cohen2006heuristics}. 
Jacobson, Luke, and Williams \cite{boundsregulatorjac}, examined bounds on $R(N)$ and $L(1, \chi_{D(N)})$, reporting results from large-scale numerical experiments. An overview of the main results concerning bounds on $L(1, \chi_{D(N)})$ and $R(N)$ is further reported in  \cite[Section 9.5]{jacobson2008solving}

Two main methods have been developed to approximate $h(N)R(N)$ using \eqref{analyticclassnumberformula}.
The first is due to Bach \cite{bach1995improved} and requires the ERH for estimating the error in the approximation, which is $O(N^{2/5 + \epsilon})$. 
This method has time complexity of $O(N^{1/5 + \epsilon})$. The second one, introduced by Srinivasan \cite{srinivasan1998computations}, approximates $h(N)R(N)$ using a technique called the `Random Summation Technique', which differs from Bach’s method. 
The error in the approximation is $O(N^{2/5 + \epsilon})$ and is estimated probabilistically in expected time $O(N^{1/5 + \epsilon})$ without assuming the ERH. However, there is a small probability that the approximation may be inaccurate, requiring recomputation.

\section*{Acknowledgments}
The first author is member of GNSAGA of INdAM and acknowledges support from Ripple’s University Blockchain Research Initiative. The second author is partially supported by project SERICS (PE00000014 - https://serics.eu) under the MUR National Recovery and Resilience Plan funded by European Union - NextGenerationEu.

\printbibliography



\end{document}